\documentclass[12pt,reqno]{amsart}

\usepackage{amssymb}

\usepackage[all]{xy}

\numberwithin{equation}{section}

\newtheorem{theorem}{Theorem}[section]
\newtheorem{proposition}[theorem]{Proposition}
\newtheorem{lemma}[theorem]{Lemma}
\newtheorem{corollary}[theorem]{Corollary}

\theoremstyle{definition}
\newtheorem{definition}[theorem]{Definition}
\newtheorem{question}[theorem]{Question}
\newtheorem{example}[theorem]{Example}

\newcommand{\cO}{\mathcal{O}}
\newcommand{\Char}{char}
\newcommand{\Proj}{\mathbb{P}}

\DeclareMathOperator*{\Div}{Div}
\DeclareMathOperator*{\Image}{Image}
\DeclareMathOperator*{\Sym}{Sym}

\begin{document}

\title[Virtual global generation in higher dimensions]{Virtual global
generation in higher dimensions}

\author[I. Biswas]{Indranil Biswas}

\address{Department of Mathematics, Shiv Nadar University, NH91, Tehsil
Dadri, Greater Noida, Uttar Pradesh 201314, India}

\email{indranil.biswas@snu.edu.in, indranil29@gmail.com}

\author[M. Kumar]{Manish Kumar}

\address{Statistics and Mathematics Unit, Indian Statistical Institute,
Bangalore 560059, India}

\email{manish@isibang.ac.in}

\author[A. J. Parameswaran]{A. J. Parameswaran}

\address{Kerala School of Mathematics, Kunnamangalam PO, Kozhikode, Kerala, 673571, India}

\email{param@ksom.res.in}



\begin{abstract}
 The notion of virtual global generation (VGG) for a vector bundle has multiple possible generalization from the case of curves to higher dimensional normal projective varieties. We study relationship between these notions. All these notions agree for curves but in higher dimension we show that this is not the case.
\end{abstract}

\maketitle

\section{Introduction}

Let $E$ be a vector bundle on an irreducible smooth projective curve $X$ defined over an algebraically 
closed field $k$.
It is called \textit{virtually globally generated} (VGG) if there is a nonconstant morphism
$$
f\, :\,Y\, \longrightarrow\, X
$$
from an irreducible smooth projective curve $Y$ such that $f^*E$ is generated by its global sections. It
was shown in \cite[Theorem 1.1]{BP2} that $E$ is VGG if and only if $E$ (respectively, $F^{n*}E$ for some 
$n\,\ge\, 1$) is a sum of a finite vector bundle and ample vector bundle if $\Char(k)\,=\,0$ (respectively,
$\Char(k)\,=\,p>0$, where $F\,:\,X\,\longrightarrow\, X$ is the absolute Frobenius morphism).

When $X$ is a normal irreducible projective variety of arbitrary dimension it is not difficult to see that a 
vector bundle $E$ on $X$ is globally generated if and only if the tautological line bundle 
$\cO_{\Proj(E)}(1)$ on the projective bundle $\Proj(E)$ over $X$ is globally generated. Moreover, this is 
also equivalent to the statement that $i^*E$ is globally generated for every closed immersion $i\,:\,C\,
\longrightarrow\, X$, where $C$ is a smooth projective curve (see Lemma \ref{GG}).

Inspired by this we give several definition of VGG (see Definition \ref{defVGG}). We show that $E$ is 
strongly VGG (globally generated after the pullback by a finite dominant map) implies $E$ is VGG (i.e., 
$\cO_{\Proj(E)}(N)$ is globally generated for some $N$). And we also show that $E$ is VGG implies $E$ is 
curve VGG (see Theorem \ref{main}). If $E$ is a direct sum of line bundles it is also shown that if $E$ is VGG 
then $E$ is strongly VGG (see Proposition \ref{p1}). But we do not know if this is true or false in 
general.

In the last section we provide examples to show that there are curve VGG bundles which are not VGG. Note that a 
vector bundle is swept by curves if it is curve VGG (follows from Lemma \ref{lem0}). We give an example to 
show that the converse is false.

\section{$m$--VGG vector bundles}

Let $k$ be an algebraically closed field.
Take an irreducible normal projective variety $X$ over $k$ of dimension $n\,\ge\, 1$, and let $E$
be a vector bundle over $X$.

\begin{definition}\label{defVGG}\mbox{}
\begin{enumerate}
\item For an integer $m\,\ge\, 1$, we will say that $E$ is $m$--virtually globally generated ($m$--VGG for short) if
for any irreducible reduced closed subscheme $Z$ of $X$ of dimension $m$, there exists a finite morphism
$f\,:\,Y\,\longrightarrow\, X$ with $\Image(f)\,=\,Z$, where $Y$ is an irreducible normal projective variety
of dimension $m$, such that the pullback $f^*E$ is generated by its global sections.

\item The vector bundle $E$ is called strongly VGG if it is $n$--VGG, where $n$, as before, is the dimension of $X$.

\item The vector bundle $E$ is called \textit{curve VGG} if it is $1$--VGG.

\item The vector bundle $E$ is called VGG if there is an integer $N\,\ge\, 1$ such that the line bundle
${\cO}_{{\mathbb P}(E)}(N)\,:=\, {\cO}_{{\mathbb P}(E)}(1)^{\otimes N}
\, \longrightarrow \, {\mathbb P}(E)$ is generated by its global sections.

\item The vector bundle $E$ is said to be \emph{swept by curves} if for every closed point $v\,\in\, E$, 
there exists a nonconstant morphism $f\,:\,Y\,\longrightarrow\, E$ from an
irreducible smooth projective curve $Y$ such that $v\,\in\, f(Y)$.
\end{enumerate}
\end{definition}

\begin{lemma}\label{lem0}
A vector bundle $E$ is swept by curves if and only if for every closed point $x\,\in\, X$, 
there exists a finite morphism $f\,:\,Y\,\longrightarrow\, X$ from a smooth projective curve $Y$, such
that $x$ is in the image of $f$ and the vector bundle $f^*E$ is generated by its global sections.
\end{lemma}

\begin{proof}
First assume that for every closed point $x\,\in\, X$, 
there exists a finite morphism $f\,:\,Y\,\longrightarrow\, X$ from a smooth projective curve $Y$, such
that $x$ is in the image of $f$ and the vector bundle $f^*E$ is generated by its global sections.
Take a point $x\, \in\, X$ and any $v\, \in\, E_x$. Take $f$ as above such
that $x$ is in the image of $f$ and the vector bundle $f^*E$ is generated by its global sections.
Fix a point $y\, \in\, Y$ for which $f(y)\,=\, x$, and also take the unique $w\, \in\, (f^*E)_y$ satisfying the
condition that the natural map
\begin{equation}\label{epsi}
\psi \ :\ f^*E\ \longrightarrow\ E
\end{equation}
sends $w$ to $v$. Choose a section
$$
s\ : \ Y \ \longrightarrow\ f^*E
$$
such that $s(y) \,=\, w$; recall that $f^*E$ is generated by its global sections. Consider the map
$$
\widehat{s} \ :=\ \psi \circ s \ :\ Y\ \longrightarrow\ E,
$$
where $\psi$ is the map in \eqref{epsi}. We have $\widehat{s}(y)\,=\, v$. Consequently,
$E$ is swept by curves.

To prove the converse, assume that $E$ is swept by curves. Take a point $x\, \in\, X$. Fix a basis
$\{v_1, \, \cdots,\, v_r\}$ of the fiber $E_x$, where $r\,=\, {\rm rank}(E)$. For $1\, \leq\, i\, \leq\, r$, let
$$
\psi_i\ :\ Y_i\ \longrightarrow\ E
$$
be a map from a smooth projective curve $Y_i$ such that the composition of maps
$$Y_i\, \xrightarrow{\,\,\, \psi_i\,\,\,}\, E \, \longrightarrow\, X$$
is finite and $\psi_i(y_i)\,=\, v_i$ for some $y_i\in Y_i$.  Consider the component $Y'$ of the fiber product
$$
Y_1\times_X Y_2\times_X\cdots \times_X Y_r
$$
that contains the point $y'=(y_1,\, y_2,\, \cdots,\, y_r)$. Denote by
$Y$ the desingularization of $Y'$, and let $\phi\, :\, Y\, \longrightarrow\, X$ be the natural
projection. For every point $y\,\in\, Y$ that maps to $y'$, it follows --- from the construction of
$Y'$ --- that the evaluation map
$$
H^0(Y,\, \phi^*E)\ \longrightarrow\ E_x,\ \ \, s\ \longrightarrow\ s(y),
$$
is surjective. So there is a finite subset $\{z_1,\, \cdots,\, z_b\}\, \subset\, X$ (which may be empty;
in other words, $b\,=\, 0$), such that the evaluation map
$$
H^0(Y,\, \phi^*E)\ \longrightarrow\ E_z,\ \ \, s\ \longrightarrow\ s(z),
$$
is surjective for all $z\, \in\, Y \setminus \phi^{-1}(\{z_1,\, \cdots,\, z_b\})$.

For each $1\, \leq\, j\, \leq\, b$, let
$$
\gamma_j\ :\ M_j\ \longrightarrow\ E
$$
be a map from a smooth projective curve $M_j$ such that
\begin{itemize}
\item the composition of maps
$$M_j\, \xrightarrow{\,\,\, \gamma_j\,\,\,}\, E \, \longrightarrow\, X$$
is finite, and

\item for all $m\, \in\, M_j$ lying above $z_j$, the evaluation map
$$
H^0(M_j,\, \gamma_j^*E)\ \longrightarrow\ E_{z_j},\ \ \, s\ \longrightarrow\ s(m),
$$
is surjective.
\end{itemize}
The curve $M_j$ can be constructed exactly as $Y$ is done after substituting $z_j$ in place
of $x$. Now consider a component $Z'$ of the fiber product
$Y\times_X M_1\times_X \cdots \times_X M_b$.
Let $Z$ be a desingularization of $Z'$. Let $f\, :\, Z\, \longrightarrow\, X$ be the natural projection.
It is evident that $f^*E$ is generated by its global sections.
\end{proof}

\begin{lemma}\label{GG}
Let $E$ be a vector bundle over an irreducible normal projective variety $X$ over $k$ of dimension
$n\,\ge \,1$. The following statements are equivalent:
\begin{enumerate}
\item The vector bundle $E$ is globally generated.

\item The line bundle ${\cO}_{{\mathbb P}(E)}(1)$ on ${\mathbb P}(E)$ is globally generated.

\item For any closed immersion from a curve $i\,:\,C \,\longrightarrow\, X$, the pulled
back vector bundle $i^*E$ is globally generated.
\end{enumerate}
\end{lemma}

\begin{proof}
Since $H^0({\mathbb P}(E),\, {\cO}_{{\mathbb P}(E)}(1))\,=\, H^0(X,\, E)$, we conclude that the statement
(1) implies the statement (2). To prove that the statement (2) implies the statement (1), take any
point $x\, \in\, X$, and consider the evaluation map
$$
\phi_x\ :\ H^0(X,\, E) \ \longrightarrow\ E_x,\ \ \, s \ \longmapsto\ s(x).
$$
If $\phi_x(H^0(X,\, E))$ is contained in a hyperplane ${\mathcal H}\, \subset\, E_x$, then consider the
point $y\, \in\, {\mathbb P}(E)_x$ corresponding to the hyperplane ${\mathcal H}$. Since
$\phi_x(H^0(X,\, E))\, \subset\, {\mathcal H}$, the evaluation map
$$
H^0({\mathbb P}(E),\, {\cO}_{{\mathbb P}(E)}(1)) \ \longrightarrow\ {\cO}_{{\mathbb P}(E)}(1)_y,
\ \ \, s \ \longmapsto\ s(y)
$$
is the zero homomorphism. In that case, ${\cO}_{{\mathbb P}(E)}(1)$ on ${\mathbb P}(E)$ is not
globally generated. Hence the statement (2) implies the statement (1).

Statement (1) evidently implies statement (3). To prove that the statement (3) implies the statement (1),
note that we can assume that $n\, \geq\, 2$, because (3) implies (1) if $X$ is a curve.
We recall Serre vanishing theorem. Let $V$ be a vector bundle on a normal projective variety $M$ over $k$,
with $\dim M \, >\, 1$, and take a very ample line bundle $L$ on $M$. Then there is a positive integer $m_0$
such that for $i\,=\, 0,\, 1$,
\begin{equation}\label{e1}
H^i(M,\, V\otimes (L^*)^{\otimes m}) \ =\ 0
\end{equation}
for all $m\, \geq\, m_0$. For any $D\, \in\, \big\vert L^{\otimes m}\big\vert$, consider the short exact
sequence
$$
0\, \longrightarrow\,V\otimes {\mathcal O}_M(-D)\,=\, V\otimes (L^*)^{\otimes m}
\, \longrightarrow\, V \, \longrightarrow\, V\big\vert_D \, \longrightarrow\, 0.
$$
It produces the following exact sequence of cohomologies
\begin{equation}\label{e2}
H^0(M,\, V\otimes (L^*)^{\otimes m}) \, \longrightarrow\,H^0(M,\, V) \,
\stackrel{\varphi}{\longrightarrow}\, H^0(D,\, V\big\vert_D)
\end{equation}
$$
\ \longrightarrow\, H^1(M,\, V\otimes (L^*)^{\otimes m}).
$$
Now from \eqref{e1} it follows that the homomorphism $\varphi$ in \eqref{e2} is an
isomorphism for all $m\, \geq\, m_0$.

Using this inductively we conclude that $H^0(M,\, V) \, =\, H^0(C,\, V\big\vert_C)$, where $C$ is a complete
intersection curve in $M$ of sufficiently large multi-degree. On the other hand, given a point of $M$, there 
are complete intersection curves of sufficiently large multi-degrees passing through it. Combining these it 
follows that the statement (3) implies the statement (1).
\end{proof}

The following theorem will be proved in Section \ref{sec3}.

\begin{theorem}\label{main}
Take an irreducible normal projective variety $X$, and let $E$ be a vector bundle over $X$. The following
two statements hold:
 \begin{enumerate}
 \item If $E$ is strongly VGG then $E$ is VGG.

 \item If $E$ is VGG then $E$ is curve VGG.
 \end{enumerate}
\end{theorem}

As a consequence of Theorem \ref{main} we obtain the following.

\begin{corollary}
Let $X$ be an irreducible smooth projective curve, and take a vector bundle $E$ over
$X$. The following four statements are equivalent.
 \begin{enumerate}
 \item $E$ is strongly VGG.
 \item $E$ is VGG.
 \item $E$ is curve VGG.
 \item $E$ is swept by curves.
 \end{enumerate}
\end{corollary}

\begin{proof}
Since $\dim X\,=\, 1$, statements (1) and (3) are equivalent by definition.

{}From Lemma \ref{lem0} it follows that the statements (3) and (4) are equivalent.

{}From Theorem \ref{main}(1) it follows that the statement (1) implies the statement (2).

{}From Theorem \ref{main}(2) it follows that the statement (2) implies the statement (3).
\end{proof}

\section{The case of line bundles}

\begin{theorem}\label{main-lb}
Let $X$ be an irreducible normal projective variety of dimension $n$, and let $L$ be a line bundle over $X$.
Then $L$ is strongly VGG if and only if $L^{\otimes N}$ is globally generated for some $N\, \geq\, 1$.
\end{theorem}

\begin{proof}
First assume that $L$ is strongly VGG. Take a pair $(Y,\, f)$, where $Y$ is
an irreducible normal projective variety of dimension $n$, and $f\,:\,Y\,\longrightarrow\, X$ is a finite
dominant morphism, such that the pullback $f^*L$ is globally generated. If the characteristic of the base field
$k$ is zero then $f$ is also generically smooth. If $\Char(k)\,=\,p\,>\,0$ then we will first reduce to the
generically smooth situation.

Denote by $K$ the separable closure of the function field $k(X)$ in $k(Y)$, and
denote by $Z$ the normalization of $X$ in $K$. Then $f$ is the composition of a generically smooth and
purely inseparable morphisms
\begin{equation}
 Y\ \stackrel{f_i}{\longrightarrow}\ Z\ \stackrel{f_s}{\longrightarrow}\ X.
\end{equation}

As $f_i$ is purely inseparable, it factors through $F^r$ for some $r$, where $F\,:\,Z\,\longrightarrow\, Z$
is the absolute Frobenius morphism, i.e. $F^r=f_i\circ h$ for a morphism $h\,:\,Z\,\longrightarrow\, Y$.
The pullback of a globally generated line bundle is also globally generated, $h^*f^*L$ is globally generated.
In view of this, replacing $f_i$ by $F^r$
and $f_s$ by its Galois closure, we may assume that $f_s$ is a (ramified) Galois covering, $f_i=F^r$ and $Y=Z$.
Also, since Frobenius commutes with every morphism, we obtain the following diagram
\begin{equation}
\xymatrix{
Y\ar[r]^{f_s}\ar[d]_{F^r} & X\ar[d]^{F^r}\\
Y\ar[r]_{f_s} & X
 }
\end{equation}
for which we may assume that $f\,=\,F^r\circ f_s\,=\,f_s\circ F^r$.

Since $F^{r*}L\,=\,L^{\otimes p^r}$, it is enough to show that $(F^{r*}L)^{\otimes N}$ is globally generated
for some $N$. Hence we may assume that $f$ itself is a Galois cover.
The Galois group
$\text{Gal}(f)$ will be denoted by $G$. Denote
\begin{equation}\label{en}
N\ :=\ |G|! .
\end{equation}
We will show that the line bundle $L^{\otimes N}$, where $N$ is defined in \eqref{en}, is globally generated.

Take a point $x\,\in\, X$, and take any point $y\,\in\, f^{-1}(x)$. Since $f^*L$ is globally generated, there
exists a global section $s$ of $f^*L$ such that $s(y)\,\ne \,0$. Note that $G$ acts on $H^0(Y,\,f^*L)$. For
every $g\,\in \,G$, let $s_g$ be the global section $g\cdot s$ of $f^*L$. Fix an ordering on $G$. Let
$l\,=\, \#\{g\,\in\, G\,\,\big\vert\,\, s_g(y)\,\ne\, 0\}$. We define a section
$$s_{l}\ =\,\ \sum_{g_1< g_2< \ldots\ <g_l} s_{g_1}\otimes s_{g_2}\otimes\cdots \otimes s_{g_{l}}$$
of $(f^*L)^{\otimes l}$. Note that $s_{l}(y)\,\ne\, 0$ as only one term in the above sum is nonzero at $y$. Also, since $f^*L$ is a line bundle, $s_{g_1}\otimes s_{g_2}=s_{g_2}\otimes s_{g_1}$. Hence $s_{l}$ is fixed by the action of $G$. Now the section
$$\widehat{s}\ =\ (s_l)^{\otimes N/l}$$
of $(f^*L)^{\otimes N}$ is fixed by the action of $G$, and $\widehat{s}(y)\,\ne\, 0$. Consequently,
$\widehat s$ descends to a section $\widetilde s$ of $L^{\otimes N}$ and $\widetilde{s}(x)\,\ne\, 0$. From
this it follows that $L^{\otimes N}$ is base point free, and hence $L^{\otimes N}$ is globally generated.

To prove the converse, assume that $L^{\otimes N}$ is globally generated for some $N\,\ge \,1$. If $\Char(k)\,=
\,p\,>\,0$, set $N\,=\,p^rN'$, where $N'$ is coprime to $p$. Then $L^{\otimes N}\,=\,(F^{r*}L)^{\otimes N'}$
is globally generated. Also if $F^{r*}L$ is strongly VGG then so is $L$. Hence replacing $L$ by $F^{r*}L$, we
may assume $N$ is coprime to $p$ when $\Char(k)\,=\,p\,>\,0$.

Fix global sections $s_1,\,\cdots,\, s_m$ of $L^{\otimes N}$ such that
 $$\cap_{i=1}^m \Div(s_i)\ =\ \emptyset.$$
 By \cite[Lemma 3.5]{BP2} (see also ``Second proof'' of \cite[Theorem 2]{Mum}), there exist cyclic covers
$f_i\,:\,Z_i\,\longrightarrow\, X$ and sections $s'_i\,\in\, H^0(Z_i,\, f_i^*L)$ such that $f_i(\Div(s'_i))
\,=\,\Div(s_i)$ for every $1\,\le\, i \,\le\, m$. Let $f\,:\,Z\,\longrightarrow\, X$ be the normalization of
a dominating component of the fiber product $Z_1\times_X Z_2\times_X\ldots\times_X Z_m$. Then the pullbacks
$\widehat{s}_i$ of $s'_i$ to $Z$ are sections of $f^*L$ with the property that
 $$\cap_{i=1}^m f(\Div(\widehat{s}_i)\ =\ \emptyset .$$
 Since $f$ is a proper map, this implies that $\cap_{i=1}^m\Div(\widehat{s}_i) \,=\,\emptyset$. So
$f^*L$ is base-point free, and hence $f^*L$ is globally generated.
\end{proof}

\section{Proof of Theorem \ref{main}}\label{sec3}

First, statement (1) will be proved.
Take a strongly VGG vector bundle $E$ on $X$. Let $Y$ be an irreducible normal projective variety with
$\dim Y\,=\, \dim X\,=\, n$ and $f\,:\,Y\,\longrightarrow\, X$ a dominant finite morphism,
such that
$f^*E$ is globally generated. Consider the projective bundles $P_X\, :=\, \Proj_X(E)$ and $P_Y\,:=\,
\Proj_Y(f^*E)$. Note that the base
change of $f$ to $P_X\,\longrightarrow\, X$ is a finite morphism
\begin{equation}\label{epf}
Pf\ :\ P_Y\ \longrightarrow\ P_X.
\end{equation}
Since $f^*E$ is globally generated, and we have a surjection $\pi^*_Y (f^*E)\,\longrightarrow\, \cO_{P_Y}(1)$,
where $\pi_Y\, :\, P_Y \, \longrightarrow\, Y$ is the natural projection,
it follows that $\cO_{P_Y}(1)$ is also globally generated. Since $\cO_{P_Y}(1)\,=\,(Pf)^*\cO_{P_X}(1)$,
where $Pf$ is the map in \eqref{epf},
we obtain that $\cO_{P_X}(1)$ is strongly VGG. Now by Theorem \ref{main-lb}, the line bundle $\cO_{P_X}(N)$
is globally generated for some $N\,\ge\, 1$. This proves the statement (1).

To prove statement (2), let $E$ be a VGG vector bundle on $X$. Take an integer $N\, \geq\, 1$ such that
$\cO_{P_X}(N)$ is globally generated, where $P_X\,=\,\Proj_X(E)$ is the projective bundle. Let
$g\,:\,C\,\longrightarrow\, X$ be a finite morphism with $C$ being a smooth projective curve.
We have $$P_C\ :=\ \Proj_C(g^*E)\ =\ C\times_X P_X.$$

The line bundle ${\cO}_{P_C}(N)$ is globally generated because it is the pullback of the
globally generated bundle $\cO_{P_X}(N)$. Note that $\pi_*\cO_{P_C}(N) \,=\,\Sym^N(g^*E)$, where
$\pi\,:\,P_C\,\longrightarrow\, C$ is the natural projection.
Since ${\cO}_{P_C}(N)$ is globally generated, it follows that
$\Sym^N(g^*E)$ is globally generated. As $\Sym^N(g^*E)$ is globally generated, we have
$\mu_{min}(\Sym^N(g^*E))\, \ge \, 0$, and hence $\mu_{min}(g^*E)\,\ge\, 0$. Now using the criterion for the curve case \cite[Theorem 1.1]{BP2}, we obtain that $g^*E$ is virtually globally generated. Hence $E$ is curve VGG. 
This completes the proof of Theorem \ref{main}.

\begin{question}\label{q1}
Does VGG implies strongly VGG?
\end{question}

Note that by Theorem \ref{main-lb}, Question \eqref{q1} has a positive answer for line bundles. The following
proposition shows that the same is true for direct sums of line bundles.

\begin{proposition}\label{p1}
Let $E$ be a direct sum of line bundles on a normal projective variety. If $E$ is VGG then it is strongly VGG.
\end{proposition}

\begin{proof}
Let $E\,=\,\bigoplus_{i=1} ^r L_i$. Fix any $i$ between 1 and $r$. The projection $E\,\longrightarrow\, L_i$ gives 
a section $g:X\,\longrightarrow\, P_X:=\Proj_X(E)$. Moreover, $g^*\cO_{P_X}(1)\cong L_r$. Since $E$ is VGG, 
$\cO_{P_X}(N)$ is globally generated for some $N$. Hence $L_i^{\otimes N}=g^*\cO_{P_X}(N)$ is globally 
generated. Theorem \ref{main-lb} ensures that there exists a finite dominant morphism $f_i\,:\,
Y_i\,\longrightarrow\, X$ such that $f_i^*L_i$ is globally generated. Let $f:Y\,\longrightarrow\, X$ be the 
normalization of a dominating component of the fiber product $Y_1\times_X\ldots\times_X Y_r$. Then $f^*L_i$ 
is globally generated for all $1\,\le\, i \,\le \,r$. Consequently, $f^*E$ is globally generated.
\end{proof}

\section{Examples}

\begin{example}\label{lex1}

We will describe an example showing that curve VGG does not imply VGG. We recall a construction
of Mumford (see \cite[page 326, Example 2]{Kl}, \cite[Ch.~I, Ex. (10.6)]{Ha}).

Let $M$ be a smooth complex projective curve of genus $g$, with $g\, \geq\, 2$. The fundamental group
$\pi_1(M,\, x_0)$ is isomorphic to the quotient of the free group generated by $a_1,\, a_2,\, \cdots,\, a_g,\,
b_1,\, b_2,\,\, \cdots, b_g$ by the normal subgroup generated by the element $\prod_{i=1}^g a_ib_ia^{-1}_i
b^{-1}_i$. So the free group generated by $a_1,\, a_2,\, \cdots, a_g$ is a quotient of $\pi_1(M,\, x_0)$.
There are $g$ elements $A_1,\, \cdots,\, A_g$ of $\text{SU}(2)$ such that subgroup of $\text{SU}(2)$
generated by $A_1,\, \cdots,\, A_g$ is dense in $\text{SU}(2)$. Fix a homomorphism
$$
\rho\ :\ \pi_1(M,\, x_0)\ \longrightarrow\ \text{SU}(2)
$$
such that $\overline{{\rho}(\pi_1(M,\, x_0))} \,=\, \text{SU}(2)$. For example, take $\rho$ such that
$\rho(a_i)\, =\, A_i$ and $\rho(b_i)\, =\, I$ for all $1\, \leq\, i\, \leq\, g$. This $\rho$ gives a holomorphic
vector bundle $V$ of rank 2 on $M$ equipped with a flat connection $\nabla$. This vector bundle $E$ is stable
\cite{NS}.

Let $\Phi\, :\, Y\, \longrightarrow\, M$ be a nonconstant morphism from a smooth complex projective curve $Y$.
Then the corresponding homomorphism of fundamental groups $\Phi_*\, :\, \pi_1(Y,\, y_0)\, \longrightarrow\,
\pi_1(M,\, x_0)$, where $y_0\, \in\, \Phi^{-1}(x_0)$, has the property that its image is a finite index
subgroup of $\pi_1(M,\, x_0)$. So the connection $\Phi^*\nabla$ has a dense monodromy, and hence $\Phi^*E$ is stable
\cite{NS}. This implies that the restriction of ${\mathcal O}_{{\mathbb P}(E)}(1)$
to every curve in ${\mathbb P}(E)$ is ample. Hence $E$ is ample. Any ample vector bundle on
$M$ is curve VGG \cite[p.~46, Theorem 3.6]{BP2}. So, $E$ is curve VGG. For any $d\, \geq\,1$, consider the
flat connection on vector bundle $\text{Sym}^d(E)$ induced by the flat connection $\nabla$ on $E$.
The closure of the monodromy of this induced connection on $\text{Sym}^d(E)$ is the subgroup
$\text{SU}(2)\, \subset\, \text{SU}(d+1)$ given by the natural action of $\text{SU}(2)$ on $\text{Sym}^d({\mathbb C}^2)$.
This implies that the vector bundle $\text{Sym}^d(E)$ is stable \cite{NS}. Hence
we have
$$
H^0(M,\, \text{Sym}^d(E))\ =\ 0.
$$
So $H^0({\mathbb P}(E),\, {\mathcal O}_{{\mathbb P}(E)}(d))\,=\, H^0(M,\, \text{Sym}^d(E))\, =\, 0$. Thus
$E$ is not VGG.
\end{example}

\begin{example}\label{lex2}

The construction in Example \ref{lex1} can be generalized in a straightforward way. Fix two integers
$m\, \geq\, 1$ and $n\,\geq\, m+1$. We will construct a smooth projective variety $X$ of dimension $n$, and a
line bundle $L$ on $X$, such that $L$ is $m$--VGG but it is not $(m+1)$--VGG.

Take a smooth complex projective curve $M$ of genus $g$, with $g\, \geq\, 2$. Fix a homomorphism
$$
\rho\ :\ \pi_1(M,\, x_0)\ \longrightarrow\ \text{SU}(m+1)
$$
such that $\overline{{\rho}(\pi_1(M,\, x_0))} \,=\, \text{SU}(m+1)$. This $\rho$ gives a holomorphic
vector bundle $E\, \longrightarrow\, M$ of rank $m+1$ equipped with a flat connection $\nabla$. The vector bundle
$E$ is stable \cite{NS}. Consider the projective bundle ${\mathbb P}(E)\, \longrightarrow\,
M$ associated to $E$. Fix a smooth projective variety $Z$ of dimension $n-m-1$. Set
$$
X\ :=\ {\mathbb P}(E)\times Z.
$$
Let
$$
\varphi\ :\ X \ =\ {\mathbb P}(E)\times Z \ \longrightarrow \ {\mathbb P}(E)
$$
be the natural projection. Consider the line bundle
$$
L \ :=\ \varphi^* {\mathcal O}_{{\mathbb P}(E)}(1) \ \longrightarrow\ X.
$$
This $L$ is $m$--VGG but it is not $(m+1)$--VGG. Note that $\dim Z \,=\, n$.
\end{example}

\begin{example}
We will give an example of a vector bundle swept by curves which is not curve VGG.

Take a projective variety $X$ with a very ample line bundle $L$ on it. Take a projective curve $Y$ with
a line bundle $L'$ on it of negative degree. Let $p_X$ (respectively, $p_Y$) be the natural projection
of $X\times Y$ to $X$ (respectively, $Y$). Consider the line bundle
$$
(p^*_X L)\otimes (p^*_Y L')\ \longrightarrow X\times Y.
$$
It is swept by curves because its restriction to $X\times \{y\}$ is very ample for every $y\, \in\, Y$.
On the other hand, the restriction of $(p^*_X L)\otimes (p^*_Y L')$ to $\{x\}\times Y$ is of negative
degree for all $x\, \in\, X$. Hence $(p^*_X L)\otimes (p^*_Y L')$ is not curve VGG.
\end{example}

\end{document}